\documentclass{amsart}%
\usepackage{amsmath}
\usepackage{amssymb}
\usepackage{amsfonts}
\usepackage{upref,amsthm,amsxtra,exscale}
\usepackage{cite}
\usepackage[colorlinks=true,urlcolor=blue,
citecolor=red,linkcolor=blue,linktocpage,pdfpagelabels,
bookmarksnumbered,bookmarksopen]{hyperref}
\usepackage{epsfig,graphics,color}
\usepackage{graphicx}%
\newtheorem{theorem}{Theorem}[section]
\theoremstyle{plain}

\newtheorem{lemma}[theorem]{Lemma}

\newtheorem{Remark}{Remark}

\numberwithin{equation}{section}

\begin{document}
\title[Spectral Theory Approach]{Spectral Theory Approach for a Class of Radial
Indefinite Variational Problems 
}
\author{Liliane A. Maia}
\address{Departamento de Matem\'{a}tica, UNB, 70910-900 Bras\'{\i}lia, Brazil.}
\email{lilimaia@unb.br}
\author{Mayra Soares}
\address{Departamento de Matem\'{a}tica, UNB, 70910-900 Bras\'{\i}lia, Brazil.}
\email{ssc\_mayra@hotmail.com}
\thanks{Research supported by FAPDF 0193.001300/2016, CNPq/PQ 308378/2017-2, PROEX/CAPES and Bolsas de Doutorado CNPq, Brazil.}
\date{\today}

\begin{abstract}
Considering the radial nonlinear Schr\"odinger equation
$$
-\Delta u + V(x)u = g(x,u) \text{ \ in \ } \mathbb{R}^N, \ N\geq3 
\eqno{({P}_r)}
$$
we aim to find a radial nontrivial solution for it, where $V$ changes sign ensuring problem $({P}_r)$ is indefinite and $g$ is an asymptotically linear nonlinearity. We work with variational methods associating problem $({P}_r)$ to an indefinite functional in order to apply our Abstract Linking Theorem for Cerami sequences  in \cite{MS} to get a non-trivial critical point for this functional. Our goal is to make use of spectral properties of operator $A:= \Delta + V(x)$ restricted to $H^1_{rad}(\mathbb{R}^N)$, the space of radially symmetric functions in $H^1(\mathbb{R}^N)$, for obtaining a linking geometry structure to the problem and by means of special properties of radially symmetric functions get the necessary compactness.
\end{abstract}
\maketitle

\section{Introduction}

\label{sec:introduction}

\hspace{0.75cm}This paper handles the following radial nonlinear Schr\"odinger equation with a sign-changing potential and an asymptotically linear nonlinearity
\begin{equation}\label{P_r}
-\Delta u + V(|x|)u = g(|x|,u) \text{ \ in \ } \mathbb{R}^N, \ N \geq 3.
\end{equation}
\hspace{0.75cm} Our goal is  to tackle the problem dropping off the monotonicity hypothesis on the nonlinear term, namely $\dfrac{g(x,s)}{s}$ nondecreasing on $s>0$ and loosen the regularity hypotheses on $g$ and $V$. In view of this, we do not look for solutions by constrained minimization either on so called Nehari or Generalized Nehari or Pohozaev  Manifolds, as was done in \cite{BW,ES, MOR, StZh,SW, AP} and references therein. Instead, we exploit the Spectral properties of the Schr\"odinger operator $A:= -\Delta + V(|x|)$ in order to get the linking geometry of the indefinite functional associated to the elliptic equation in (\ref{P_r}). 
Since problem (\ref{P_r}) is radially symmetric, to deal with the Spectral Theory of $A$ restrictive hypotheses on $V$ are not necessary. In fact, it suffices to request informations under an associated operator $\bar{A}$ on the half-line, which is more manageable. Hence we assume that the potential $V$ satisfies:\\

\hspace{-0.5cm}$(V_1)_r  \ V \in L^\infty(\mathbb{R}^N)$ \ is a radial sign-changing function, $V(x)=V(|x|)=V(r), \ r\geq0$;
% \ such that \ $V(x)=V_1(x)+V_2(x)$,\  where $V_1(x)= V_1(|x|)$ is $2\pi$-periodic nonzero sign-changing function and $V_2(x) = V_2(|x|)= o(1)$, as $|x|\to +\infty.$
%$\displaystyle\lim_{|x|\to +\infty}V(x)=V_{\infty};$

%\hspace{-0.7cm}$(V_2)$ \ Setting $A := -\Delta + V(x)$, a self-adjoint operator of $L^2(\mathbb{R}^N)$,
%$$\sup\left[\sigma(A)\cap(-\infty,0)\right]= \sigma^{-} < 0 < \sigma^{+} = \inf\left[\sigma(A)\cap(0,+\infty)\right].$$

\hspace{-0.5cm}$(V_2)_r$  Setting $\bar{V}(r)= V(r)+\dfrac{(N-1)(N-3)}{4r^2}$ and $\bar{A} := -\dfrac{d^2}{dr^2} + \bar{V}(r)$, an operator of $L^2(0,\infty)$, \  $0\notin \sigma_{ess}(\bar{A})$ and
$$\sup\left[\sigma(\bar{A})\cap(-\infty,0)\right]= \sigma^{-} < 0 < \sigma^{+} = \inf\left[\sigma(\bar{A})\cap(0,+\infty)\right].$$

\hspace{0.75cm}Moreover, we take the nonlinearity $g$ under the hypotheses:\\

\hspace{-0.5cm}$(g_1) \ \ g(x,s)\in C(\mathbb{R}^N\times \mathbb{R},\mathbb{R})$ is a radial function such that $\displaystyle\lim_{|s| \to 0}\dfrac{g(x,s)}{s}=0$, uniformly in $x$ and for all $t \in \mathbb{R}$,
$$G(x,t)= \displaystyle\int_0^tg(x,s)ds \geq 0;$$
$(g_2) \ \ \displaystyle\lim_{|s| \to +\infty}\dfrac{g(x,s)}{s}= h(x),$ uniformly in $x$, where $h \in L^{\infty}(\mathbb{R}^N)$;\\
\hspace{-0.5cm}$(g_3) \ \ a_0 = \displaystyle\inf_{x \in \mathbb{R}^N} h(x)>\sigma^+=\inf\left[\sigma(A)\cap(0,+\infty)\right];$\\
\hspace{-0.5cm}$(g_4) \ \ $Setting $\mathcal{O}:=A-\mathcal{H}$, where $\mathcal{H}$ is the operator multiplication by $h(x)$ in $L^2(\mathbb{R}^N)$  and denoting by $\sigma_p(\mathcal{O})$ the point spectrum of $\mathcal{O}$,
$$0\notin \sigma_p(\mathcal{O}).$$ 

%%%%%%%%%%%%%%%%%%%%%%%%%%%%%%%%%%%%%%%%%%%%%%%%%%%%%%%%%%%%%%%%

\hspace{0.75cm}The inspiration for this work came from the papers \cite{AP,StZh}. In the former, A. Azzollini and A. Pomponio treated an autonomous radial nonlinear Schr\"odinger equation with a nonlinearity under Berestycki and Lions hypotheses (cf. \cite{BeLi}). Besides their potential $V \in C^1(\mathbb{R}^N,\mathbb{R})$ satisfied some restrictions on its derivatives, and it had a non-positive limit at infinity, which ensured that $0\in \sigma_{ess}(A)$ in their case. Hence, we complement their work considering cases such that $0\notin\sigma_{ess}(A)$, furthermore, we only require $V\in L^\infty(\mathbb{R}^N)$ such that the spectrum of $A$ has a gap in $0$, which is at most an isolated point of $\sigma(A)$.

\hspace{0.75cm}C. A. Stuart and H. S. Zhou in \cite{StZh} worked with a class of radial nonlinear Schr\"odinger equations depending on $\lambda$, which is a constant potential and an asymptotically linear nonlinearity, but including the monotonicity hypothesis, as mentioned previously. They solved their class of problems by applying a variant of the Mountain Pass Theorem in \cite{BBF}. The most interesting feature in their paper was to make use of the relation between the associated problem on the half-line and the original problem in $\mathbb{R}^N$. Following their ideas, we extract spectral informations from the associated operator $\bar{A}$ on the half-line, to guarantee that problem (\ref{P_r}) satisfies the necessary conditions for the linking geometry.

\hspace{0.75cm}Another notable work, which is worth mentioning is \cite{W} by T. Watanabe. Although the author considers an autonomous radial nonlinear Schr\"odinger equation in $\mathbb{R}^2$, with positive potential and a nonlinearity with monotonicity assumption, our hypotheses are similar to his on the nonlinear term. Furthermore, as in \cite{StZh} T. Watanabe first worked with the associated problem on the half-line, which also encouraged us to make the same.

\hspace{0.75cm}In this article, since we work in $H^1_{rad}(\mathbb{R}^N)$, the novelty lies in using this idea of investigating the spectral properties of the associated operator $\bar{A}$ on the half-line, and by doing this, avoiding a deeper study of the spectral theory of the operator ${A}$ in $H^1(\mathbb{R}^N)$. Thereby, we are able to deal with much more general potentials, for instance those which do not have a limit at infinity.

\hspace{0.75cm}Next section is devoted to present the spectral properties of the operators, the choice of suitable Hilbert spaces, the variational framework and then we state our main result. In section $3$ we prove the required compactness for the associated functional. Section $4$ describes how to establish the linking geometry by means of the sharp construction of the linking components based on the spectral results. The core of our arguments is to take advantage of the strict inequality in $(g_3)$ throughout this section. Finally, in section $5$ the boundedness of Cerami sequences for the functional is obtained and a proof for the main result is presented.

\section{The Variational Setting}

\hspace{0.75cm}Considering $A:=-\Delta+V(x)$ as an operator of $L^2(\mathbb{R}^N)$, since $V \in L^\infty(\mathbb{R}^N)$, $A$ as well as $\bar{A}$ are self-adjoint operators. Due to Hardy's Inequality, operator $\bar{A}$ is treated in $H^1_0(0, \infty)$, which can be written as $H^1_0(0,\infty)=H^-\oplus H^0 \oplus H^+$, with $H^-, \ H^0, \ H^+$ the subspaces of $H^1_0(0,\infty)$ where $\bar{A}$ is respectively negative, null and positive definite. In view of $(V_2)_r$ each $u \in H^+$ satisfies $$\sigma^+||u||^2_{L^2(0,\infty)}\leq (\bar{A}u,u)_{L^2(0,\infty)}.$$ Moreover, given $u\in H^1_0(0,\infty)$ and setting $w:=r^{\frac{1-N}{2}}u$, it yields $w\in H^1_{rad}(\mathbb{R}^N)$ (cf. \cite{StZh} section 3), where $H_{rad}^1(\mathbb{R}^N)$ is the subspace of the radially symmetric functions in $H^1(\mathbb{R}^N)$. In addition, changing variables, $w$ satisfies
$$||w||^2_{L^2(\mathbb{R}^N)} = \int_{\mathbb{R}^N}|w(x)|^2dx = \omega_N\int_0^\infty |u(x)|^2dr = \omega_N||u||^2_{L^2(0,\infty)},$$
and
\begin{eqnarray*}
	(Aw,w)_{L^2(\mathbb{R}^N)}&=& \int_{\mathbb{R}^N}\left(|\nabla w(x)|^2 +V(x)w(x)^2\right)dx\\
	&=& \omega_N\int_0^\infty \left(|u'(r)|^2 +\bar{V}(r)u(r)\right)dr\\
	&=& \omega_N(\bar{A}u,u)_{L^2(0,\infty)},
\end{eqnarray*}
where $\omega_N$ is the $(N-1)$-dimensional surface measure of the sphere $S^{N-1}\subset \mathbb{R}^{N}$. Hence $\sigma^+||w||^2_{L^2(\mathbb{R}^N)}\leq(Aw,w)_{L^2(\mathbb{R}^N).}$ If some function $\tilde{w} \in H^1_{rad}(\mathbb{R}^N)$ has satisfied $$0<(A\tilde{w},\tilde{w})_{L^2(\mathbb{R}^N)}< \sigma^+||\tilde{w}||^2_{L^2(\mathbb{R}^N)},$$ by approximation it could be regarded as a smooth function \ and \ then \ setting $\tilde{u} := r^{\frac{N-1}{2}}\tilde{w}$, it would belong to $H^+$ and would satisfy $$\sigma^+||\tilde{u}||^2_{L^2(0,\infty)}>(\bar{A}\tilde{u},\tilde{u})_{L^2(0,\infty)},$$
which contradicts $(V_2)_r$. Hence, writing $H^1_{rad}(\mathbb{R}^N)=E^-\oplus E^0\oplus E^+$, with $E^-, \ E^0, \ E^+$ the subspaces where $A$ is respectively negative, null and positive definite, if $w\in E^+$ it satisfies $\sigma^+||w||^2_{L^2(\mathbb{R}^N)}\leq ({A}w,w)_{L^2(\mathbb{R}^N)}$.

\begin{Remark}\label{crr1}
	Note that if $\sigma^+$ is an eigenvalue of $\bar{A}$ with eigenfunction $u$, \ the same argument as above \ shows that $\sigma^+$ is an eigenvalue of $A$, \ with a radial \ eigenfunction $w = r^{\frac{1-N}{2}}u\in E^+$. On the other hand, if $\sigma^+$ is not an eigenvalue of $\bar{A}$, either it belongs to $\sigma_{ess}(\bar{A})$ or it is a cluster point of $\sigma(\bar{A})$, then given $\varepsilon>0$ there exist $u_{\varepsilon}\in H^+$ such that
	$$\sigma^+||u_\varepsilon||^2_{L^2(0,\infty)}< (\bar{A}u_\varepsilon,u_\varepsilon)_{L^2(0,\infty)} < (\sigma^+ +\varepsilon)||u_\varepsilon||^2_{L^2(0,\infty)},$$
	which ensures that $w_\varepsilon:= r^{\frac{1-N}{2}}u_\varepsilon \in E^+$ satisfies
	$$\sigma^+||w_\varepsilon||^2_{L^2(\mathbb{R}^N)}< ({A}w_\varepsilon,w_\varepsilon)_{L^2(\mathbb{R}^N)} < (\sigma^+ +\varepsilon)||w_\varepsilon||^2_{L^2(\mathbb{R}^N)}.$$ Therefore,
	\begin{equation}\label{crr1e1}
	\sigma^+ = \inf_{w\in E^+}\dfrac{(Aw,w)_{L^2(\mathbb{R}^N)}}{||w||^2_{L^2(\mathbb{R}^N)}}.
	\end{equation}
	Applying the same arguments comparing $H^-$ and $E^-$, it yields
	\begin{equation}\label{crr1e2}
	-\sigma^- = \inf_{w\in E^-}\dfrac{-(Aw,w)_{L^2(\mathbb{R}^N)}}{||w||^2_{L^2(\mathbb{R}^N)}}.
	\end{equation}
\end{Remark}
\hspace{0.75cm}From hypothesis $(V_2)_r$  either $0 \notin \sigma(\bar{A})$ or it is an isolated eigenvalue of $\bar{A}$.  Since by assumption  $0\notin\sigma_{ess}(\bar{A})$, if $0\in \sigma(\bar{A})$ it is an eigenvalue of finite multiplicity, hence $\ker(\bar{A})$ is finite dimensional. The same conclusions hold for $A$, since there exists a correspondence between the eigenfunctions of $\bar{A}$ and the radial eigenfunctions of $A$. Furthermore, $u_1,u_2 \in H^1_0(0,\infty)$ are orthogonal in $L^2(0,\infty)$ iff $w_1=r^{\frac{1-N}{2}}u_1$ and $w_2=r^{\frac{1-N}{2}}u_2$ are orthogonal in $L^2(\mathbb{R}^N)$. Indeed,
$$\int_0^\infty u_1(r)u_2(r)dr = \dfrac{1}{\omega_N}\int_{\mathbb{R}^N}w_1(x)w_2(x)dx.$$
Therefore, $H^i$ is infinite dimensional iff $E^i$ is infinite dimensional, for $i=-, \ 0, \ +.$

\hspace{0.75cm}A typical example of $V$ satisfying $(V_1)_r-(V_2)_r$ is a suitable continuous, periodic and sign-changing $V(r)$, such that $0\notin \sigma\Big(-\dfrac{d^2}{dr^2}+V(r)\Big)$, hence $0$ is in a gap of the spectrum, which is composed by closed intervals. Since $V(r)$ is continuous and changes sign, $-\dfrac{d^2}{dr^2}+V(r)$ has positive and negative spectrum. Moreover, $\bar{V}=V + V_N$, \ where \ $V_N(r)= \dfrac{(N-1)(N-3)}{4r^2}$ decays sufficiently fast, then it is a Kato's potential and hence $\bar{A}-$compact, which ensures $\sigma_{ess}(\bar{A})=\sigma\Big(-\dfrac{d^2}{dr^2}+V(r)\Big)$ by Weyl's theorem (cf. \cite{O} page 290 Corollary 11.3.6 and also \cite{HS} sections 14.2-14.3), thus $0\notin \sigma_{ess}(\bar{A})$ and $\sigma(\bar{A})$ also has positive and negative part. Therefore, $(V_2)_r$ is satisfied.\\

\begin{Remark}
	Simple examples of potentials which satisfy or not our assumptions:\\
	
	\hspace{-0.4cm}Ex 1. $V(r)= \cos (r)$ satisfies $(V_1)_r-(V_2)_r$ by the previous observations.\\
	
	\hspace{-0.4cm}Ex 2. $V(r)= \dfrac{1}{1+r^2} - \dfrac{1}{2}$, does not satisfy $(V_2)_r$, since $0 \in \sigma_{ess}(\bar{A}).$ In fact, $\displaystyle\lim_{r\to +\infty}V(r)= - \dfrac{1}{2}$, hence $\sigma_{ess}(\bar{A})=\sigma_{ess}(A)= [-\dfrac{1}{2},+\infty).$
	\end{Remark}

\hspace{0.75cm}An example of $g$ satisfying $(g_1)-(g_4)$ is an asymptotically linear continuous function such that $h(x)\equiv a_0>\sigma^+$ as in $(g_3)$, then for a periodic $V$, since $\sigma(A)$ is pure absolutely continuous, $a_0 \notin \sigma_p(A)$ and hence $0\notin \sigma_p(\mathcal{O})$ as in $(g_4)$. Model nonlinearities which appear in Physics of propagations of laser beans in nonlinear medium with saturations are for instance
$$g(x,s)=\dfrac{s^3}{1+ {a_0^{-1}}{s^2}} \qquad \text{and} \qquad g(x,s) = \left(a_0 - \dfrac{1}{\exp{s^2}}\right)s.$$

\begin{Remark}\label{crr2}
	Due to $(g_1)-(g_2)$, given $\varepsilon>0$ and $2\leq p \leq 2^*$ there exists a constant $C_{\varepsilon}>0$	such that
	\begin{equation}\label{cre1}
	|g(x,s)|\leq \varepsilon|s|+ C_{\varepsilon}|s|^{p-1},
	\end{equation}
	and hence
	\begin{equation}\label{cre2}
	|G(x,s)|\leq \dfrac{\varepsilon}{2}|s|^2+ \dfrac{C_{\varepsilon}}{p}|s|^p,
	\end{equation}
	for all $s \in \mathbb{R},$ and for all  $x \in \mathbb{R}^N.$
\end{Remark}

\hspace{0.75cm}The functional $I:H^1(\mathbb{R}^N)\to \mathbb{R}$ associated to problem (\ref{P_r}) is given by 
\begin{equation}\label{cre3}
I(u)= \dfrac{1}{2}(Au,u)_{L^2(\mathbb{R}^N)} - \int_{\mathbb{R}^N}G(x,u)dx.
\end{equation}
Note that, in view of $(V_1)_r$ and $(g_1)-(g_2)$ $I:H^1(\mathbb{R}^N)\to \mathbb{R}$ \ is well defined \ and $I \in C^1(H^1(\mathbb{R}^N),\mathbb{R})$. Thus, as usual, \ a weak solution for $(P_r)$ \ is a critical point \ of $I:H^1(\mathbb{R}^N)\to \mathbb{R}$, a function $u \in H^1(\mathbb{R}^N)$ such that for all $v \in H^1(\mathbb{R}^N)$
$$I'(u)v= (Au,v)_{L^2(\mathbb{R}^N)} - \int_{\mathbb{R}^N}g(x,u(x))v(x)dx = 0.$$

\hspace{0.75cm}In order to obtain a nontrivial critical point of the functional $I$ we make use of an abstract linking theorem proved by the authors in \cite{MS}, which we now recall.

\begin{theorem}\label{ALT}
\textbf{Linking Theorem for Cerami Sequences:} Let $E$ be a real Hilbert space, with inner product $\big( \cdot, \cdot \big)$, $E_1$ a closed subspace of $E$ and $E_2=E_1^{\perp}$. Let $I \in C^1(E,\mathbb{R})$ satisfying:\\

	\hspace{-0.5cm}$(I_1) \ \ I(u)=\dfrac{1}{2}\big(Lu,u\big)+B(u),$ for all $u \in E$, where $u=u_1 +u_2 \in E_1 \oplus E_2$, $Lu=L_1u_1 + L_2u_2$ and\  $L_i: E_i \rightarrow E_i, \ i=1,2$ is a bounded linear self adjoint mapping.\\
	
	\hspace{-0.5cm}$(I_2) \ \ B$ is weakly continuous and uniformly differentiable on bounded subsets of $E$.\\
	
	\hspace{-0.5cm}$(I_3) \ $ There exist Hilbert manifolds $S,Q \subset E$, such that $Q$ is bounded and has boundary $\partial Q$, constants $\alpha > \omega$ and $v \in E_2$ such that\\
	$(i) \ S\subset v + E_1$ and $I \geq \alpha$ on $S$;\\
	$(ii) \ I\leq \omega$ on $\partial Q$;\\
	$(iii) \ S$ and $\partial Q$ link.\\
	%, that is satisfy de linking definition in \ref{d1}. \\
	
	\hspace{-0.5cm}$(I_4) \  \ $
If for a sequence $(u_n)$, $I(u_n)$ is bounded and  $\left(1+||u_n||\right)||I'(u_n)|| \to 0$, as $n\to +\infty$, then $(u_n)$ is bounded.
	\vspace{0.5cm}\\
	\hspace{-0.5cm}Then I possesses a critical value $c\geq \alpha$.
\end{theorem}

\hspace{0.75cm}Since $V$ and $G$ are radial functions, in order to apply Theorem \ref{ALT}, it is convenient to define $E:=H^1_{rad}(\mathbb{R}^N)$, which is the Hilbert subspace of all radial symmetric functions in $H^1(\mathbb{R}^N)$ and consider $I:E\to \mathbb{R}$. In fact, functions in $E$ satisfy special properties that make true all necessary hypotheses on $I:E\to \mathbb{R}$, for example, recall that by Strauss \cite{S} (cf. also \cite{BeLi} Theorem A.I'.) $E$ is compactly embedded in $L^{\beta}(\mathbb{R}^N)$, for any $\beta \in (2,2^*)$.

\hspace{0.75cm}Defining $Q_A: E\to \mathbb{R}$ by
$$Q_A(u) := \displaystyle\int_{\mathbb{R}^N}|\nabla u(x)|^2dx + \int_{\mathbb{R}^N}V(x)u^2(x)dx = \dfrac{1}{2}(Au,u)_{L^2(\mathbb{R}^N)},$$
it \ is \ a \ continuous \ quadratic \ form on $E$. \ Since \ $E^0,E^-,E^+$ are \ the \ closed \ subspaces of $E$ \ on \ which \ $Q_A$ is null, negative \ and positive \ definite, \ then \ $E=E^0\oplus E^-\oplus E^+$.
Moreover, if $B_{Q_A}[u,v] =(Au,v)_{L^2(\mathbb{R}^N)}$ for all $u, \ v\in E$, 
is the bilinear form associated to $Q_A$ and $u, \ v$ belong to distinct such subspaces, then $B_{Q_A}[u,v]=0$ and $Q_A(u+v)=Q_A(u)+Q_A(v)$. In addition $E^0,E^-,E^+$ are mutually orthogonal in the $L^2(\mathbb{R}^N)$-inner product. Hence, for $u=u^0+u^++u^- \in E$, it is suitable to take as an equivalent norm in $E$ the expression
$$||u||^2 = ||u||^2_E:= ||u^0||_2^2+Q_A(u^+)-Q_A(u^-),$$
and the associated inner product, obtained by means of $B_{Q_A}[u,v]$, which makes $E$ a Hilbert space with $E^0, E^+,E^-$ orthogonal subspaces of $E$. In fact, by $(V_2)_r$ and Remark \ref{crr1}
% Spectral Theory (cf. \cite{BS}, Theorem 1.1', page 394; see also \cite{P} section 3),
%see Pankov's notes page 29 proposition 3.8
for all $u^+ \in E^+$  and  for all $u^- \in E^-$ it yields
\begin{equation}\label{cre4}
\sigma^+||u^+||^2_2\leq \int_{\mathbb{R}^N}\Big(|\nabla u^+(x)|^2+V(x)(u^+(x))^2\Big)dx = ||u^+||^2,
\end{equation}
and
\begin{equation}\label{cre5}
-\sigma^-||u^-||^2_2\leq -\int_{\mathbb{R}^N}\Big(|\nabla u^-(x)|^2+V(x)(u^-(x))^2\Big)dx = ||u^-||^2,
\end{equation}
which ensures that the norm chosen above is equivalent to the standard norm in $H^1_{rad}(\mathbb{R}^N)$, once $E^0 = \ker(A)$ is finite dimensional.
%\begin{Remark}\label{RSS}
%Note that, in virtue of $(V_2)$ neither $E^+$ nor $E^-$ is zero dimensional. Indeed, \ $E^{\pm}= H^1_{rad}(\mathbb{R}^N)\cap (H^1(\mathbb{R}^N))^{\pm}\not= \{0\}$, \ by \ \ Schwarz \ Symmetrization, since given $u^{\pm} \in (H^1(\mathbb{R}^N))^{\pm}$, its symmetrization is in $(H^1(\mathbb{R}^N))^{\pm}$ as well.
%\end{Remark}

\hspace{0.75cm}Observe that, $I(u)=Q_A(u)-\displaystyle\int_{\mathbb{R}^N}G(x,u(x))dx$, for all $u \in E$ and since $E$ is a subspace of $ H^1(\mathbb{R}^N)$,  $I \in C^1(E,\mathbb{R})$. Moreover, $I$ is indefinite on  $E$, henceforth the goal is to apply Theorem \ref{ALT} in order to get a critical point of $I$ restricted to $E$, and by applying the Principle of Symmetric Criticality (cf. \cite{Pal}) conclude the critical point is actually a critical point of $I:H^1(\mathbb{R}^N)\to \mathbb{R}$, namely a weak solution to $(P_r)$. Our main result is stated bellow.

\begin{theorem}\label{crt1}
	Suppose $(V_1)_r-(V_2)_r$ and $(g_1)-(g_4)$ hold. Then problem $(P_r)$ in (\ref{P_r}) possess a radial, nontrivial, weak solution in $H^1(\mathbb{R}^N)$.
\end{theorem}	

\hspace{0.75cm}In order \ to show that $I$ satisfies \ $(I_1)$ \ of Theorem \ref{ALT}, set $E_1:=E^+$ \ and $E_2:=E^-\oplus E^0$, then it yields $E_2^{\perp}=E_1$. Now, define $L_i:E_i\to E_i,$ for all $u \in E_i$, as given by
$$(L_iu,v)_E=Q_A'(u)v=B_{Q_A}[u,v]= (Au,v)_{L^2(\mathbb{R}^N)},$$
for all $v \in E_i, \ i=1,2,$ where $Q_A'(u)v$ denotes Fréchet derivative of $Q_A$ at $u$ acting on $v$. Hence, $L = L_1 + L_2:E_1\oplus E_2 \to E_1\oplus E_2$ is a well defined, linear, bounded operator and satisfies $$Q_A(u)= \dfrac{1}{2}(Au,u)_{L^2(\mathbb{R}^N)}=\dfrac{1}{2}Q_A'(u)u=\dfrac{1}{2}B_{Q_A}[u,u]= \dfrac{1}{2}(Lu,u)_E.$$
Thus, setting 
${B(u):= -\displaystyle\int_{\mathbb{R}^N}G(x,u(x))dx},$
for all $u \in E,$ it is possible to write $$I(u)=\dfrac{1}{2}(Lu,u)+B(u),$$ satisfying $(I_1)$.

\section{The Weak Continuity and Uniform Differentiation of I}

\hspace{0.75cm}The following lemma is a variant of Strauss compactness lemma \cite{S} (see also Theorem A.I. in \cite{BeLi}) which is essential for the proof that $I$ satisfies $(I_2)$. This version applies to functions $P$ depending also on the space variable $x$. Assuming the dependence is uniform on $x$ as $|s|$ goes to zero and infinity, the proof follows with minor changes.

\begin{lemma}\label{crl0}
	Let $P:\mathbb{R}^N\times\mathbb{R}\to \mathbb{R}$ and $Q:\mathbb{R}\to \mathbb{R}$ be two continuous functions satisfying
	\begin{equation}\label{crl0e1}
	\dfrac{P(x,s)}{Q(s)}\to 0, \ \ \ \text{uniformly \ in \ } x \ \text{as} \ |s|\to +\infty.
	\end{equation}
	Let $(u_n)$ be a sequence of measurable functions from $\mathbb{R}^N$ to $\mathbb{R}$ such that
	\begin{equation}\label{crl0e2}
	\sup_n\int_{\mathbb{R}^N}|Q(u_n(x))|dx <+\infty,
	\end{equation}
	and
	\begin{equation}\label{crl0e3}
	P(x,u_n(x)) \to v(x) \text{ \ a. \ e. \ in \ } \mathbb{R}^N,
	\end{equation}
	as $n \to +\infty.$ Then for any bounded Borel set $\mathcal{B}$ one has
	\begin{equation}\label{crl0e4}
	\int_{\mathcal{B}} |P(x,u_n(x))-v(x)|dx \to 0,
	\end{equation}
	as $n \to +\infty.$ If one further assumes that
	\begin{equation}\label{crl0e5}
	\dfrac{P(x,s)}{Q(s)} \to 0, \ \ \ \text{uniformly \ in \ } x \ \text{as} \ s \to 0,
	\end{equation}
	and 
	\begin{equation}\label{crl0e6}
	u_n(x)\to 0 \text{ \ as \ } |x|\to +\infty, \text{ \ uniformly \ with \ respect \ to \ } n,
	\end{equation}
	then $P(\cdot,u_n(\cdot))$ converges to $v$ in $L^1(\mathbb{R}^N)$ as $n\to +\infty$.
\end{lemma}
\begin{proof} In order to prove the first part of the proposition, it is sufficient to show that $P(x,u_n(x))$ is uniformly integrable on $\mathcal{B}$. In fact, if this is the case, due to (\ref{crl0e3}) $$\int_{\mathcal{B}\cap\{|P(x,u_n(x))|\leq R\}}|P(x,u_n(x))-v(x)|dx \to 0,$$ as $n\to +\infty$, by applying Lebesgue Dominated Convergence Theorem, and the integral $$\int_{\mathcal{B}\cap\{|P(x,u_n(x))| > R\}}|P(x,u_n(x))|dx,$$ is controlled by uniform integration. By condition (\ref{crl0e1}) there exists $C>0$ such that $$|P(x,u_n(x))|\leq C(1 + |Q(u_n(x))|), \ x \in \mathbb{R}^N.$$
	Thus, in view of (\ref{crl0e2}) and Fatou's Lemma, it follows that $P(\cdot,u_n(\cdot))$ and $v$ are in $L^1(\mathcal{B})$. Moreover, since $P$ is continuous, it maps compacts sets on compact sets, hence fixed $R>0$, if for some $x\in \mathbb{R}$, $|P(x,u_n(x))| > R$, there exists $M= M(R)>0$, such that $|u_n(x)|> M(R)$ and $M(R)\to +\infty$ as $R\to +\infty$. Then
	$$\int_{\mathcal{B}\cap\{|P(x,u_n(x))|> R\}}|P(x,u_n(x))|dx \leq \int_{\mathcal{B}\cap\{|u_n(x)|> M(R)\}}|P(x,u_n(x))|dx.$$
	Applying condition (\ref{crl0e1}), given $\varepsilon>0$ there exist $M(R)>0$, such that $|u_n(x)|\geq M(R)$ implies
	$|P(x,u_n(x))|\leq \varepsilon|Q(u_n(x))|$ and $\varepsilon=\varepsilon(R) \to 0$ as $M(R)\to +\infty$. Then, there exist $\tilde{C}>0$ such that
	\begin{eqnarray*}
		\int_{\mathcal{B}\cap\{|P(x,u_n(x))|> R\}}|P(x,u_n(x))|dx &\leq& \int_{\mathcal{B}\cap\{|u_n(x)| > M(R)\}}|P(x,u_n(x))|dx\\
		&\leq& \varepsilon(R)\int_\mathcal{B}|Q(u_n(x))|dx\\
		&\leq& \tilde{C}\varepsilon(R),
	\end{eqnarray*}
	which shows the uniform integrability and ensures the result.
	
	\hspace{0.75cm}For the \ second part, that \ $P(\cdot,u_n(\cdot))$ \ converges to $v$ \ in $L^1(\mathbb{R}^N)$ \ as  $n\to +\infty$, \ note that \ in virtue of \ (\ref{crl0e5}) \ given $\varepsilon>0$ \ there exists $\delta>0$ such that $|s|\leq \delta$ implies $|P(x,s)|\leq \varepsilon|Q(s)|$, uniformly in $x$. Moreover, by (\ref{crl0e6}) given $\delta>0$ there exists $R_0>0$ such that $|u_n(x)|\leq \delta$ for all $|x|\geq R_0$, uniformly in $n$. Thus, $|x|\geq R_0$ implies $|P(x,u_n(x))|\leq \varepsilon|Q(u_n(x))|$, uniformly in $n$. Therefore, by Fatou's Lemma $v \in L^1(\mathbb{R}^N)$ and
	$$\int_{\{|x|\geq R_0\}}|v(x)|dx \leq \liminf_{n \to\infty}\int_{\{|x|\geq R_0\}}|P(x,u_n(x))|dx \leq \tilde{C}\varepsilon.$$ In addition, from the first part, there exists $n_0\in \mathbb{N}$ such that for $n\geq n_0$ 
	$$\int_{\{|x|<R_0\}}|P(x,u_n(x))-v(x)|dx\leq \varepsilon.$$
	Hence, for $n\geq n_0$ it yields
	$$\int_{\mathbb{R}^N}|P(x,u_n(x))-v(x)|dx \leq (2\tilde{C}+1)\varepsilon,$$
	which gives the result.
\end{proof}	

\hspace{0.75cm}By means of the previous lemma, next result holds.

\begin{lemma}\label{crl1}
	If $g$ satisfies $(g_1)-(g_2)$, then $B$ is weakly continuous.
\end{lemma}

\begin{proof}Let \ $(u_n)\in E$ \ \ and \ suppose \ $u_n \rightharpoonup u$ in \ $E$, then $(u_n)$ is bounded in $E$. Due to $(g_1)-(g_2)$, for $2<p<2^*$ one has
	\begin{equation}\label{crl1e1}
	\lim_{s \to 0} \dfrac{G(x,s)}{|s|^2} = 0 \text{ \ and \ } \lim_{|s| \to +\infty} \dfrac{G(x,s)}{|s|^p} = 0,
	\end{equation}
	uniformly in $x$. Hence, choosing $Q(s)= |s|^2+ |s|^p$, and $P(\cdot,s)= G(\cdot,s)$, it is possible to apply Lemma \ref{crl0}. Indeed, in view of (\ref{crl1e1}) it follows that
	\begin{equation}\label{crl1e2}
	\lim_{s \to 0} \dfrac{G(x,s)}{|s|^2+|s|^p} = 0 \text{ \ and \ }
	\lim_{|s| \to +\infty}  \dfrac{G(x,s)}{|s|^2+|s|^p} = 0,
	\end{equation}
	uniformly in $x$. Then $P$ and $Q$ satisfy (\ref{crl0e1}) and (\ref{crl0e5}). Moreover, 
	\begin{equation}\label{crl1e3}
	\sup_n\int_{\mathbb{R}^N}\Big(|u_n(x)|^2 + |u_n(x)|^p\Big)dx = \sup_n\Big(||u_n||_2^2+||u_n||_p^p\Big) \leq C < +\infty,
	\end{equation}
	since $(u_n)$ is bounded in $E$ and $E$ is continuously embedded in $L^2(\mathbb{R}^N)$ and $\ L^p(\mathbb{R}^N)$. Hence (\ref{crl0e2}) is satisfied. Provided that $u_n \rightharpoonup u$ in $E$ and $E$ is compactly embedded in $L^p(\mathbb{R}^N)$, $u_n \to u$ in $L^p(\mathbb{R}^N)$ and $u_n(x) \to u(x)$ almost everywhere in $\mathbb{R}^N$. Thus, choosing $v(x)= G(x,u(x))$ \ it follows that \ (\ref{crl0e3}) \ is satisfied. \ Finally, \ since $(u_n)\subset H^1_{rad}(\mathbb{R}^N)$ and $u_n(x)\to u(x)$ \ almost everywhere \ in $\mathbb{R}^N$, it yields \ $\displaystyle\lim_{|x| \to +\infty}u_n(x)=0$, \ uniformly with \ respect to $n$ (cf. \cite{BeLi} Lemma A.II.). \ Therefore, applying Lemma \ref{crl0} it yields \ $ G(\cdot,u_n(\cdot)) = P(\cdot,u_n(\cdot)) \to v = G(\cdot,u(\cdot))$ in $L^1(\mathbb{R}^N)$ as $n \to +\infty$, namely,
	$$B(u_n) = - \int_{\mathbb{R}^N}G(x,u_n(x))dx \to - \int_{\mathbb{R}^N}G(x,u_n(x)) = B(u),$$
	as $n \to +\infty$ and then $B$ is weakly continuous.
\end{proof}

\begin{lemma}\label{crl2}
	Assume that $g$ satisfies $(g_1)-(g_2)$, then $B$ is uniformly differentiable on bounded sets of $E$.
\end{lemma}
\begin{proof}
	First, note that fixed $R>0$ and given $u+v, \ v \in B_R \subset E$, the closed ball centered on the origin, one has
		\begin{eqnarray}\label{crl2e1}
		&&|B(u+v)-B(u)-B'(u)v| \nonumber\\ &=&\int_{\mathbb{R}^N}\big|G\big(x,u(x)+v(x)\big)- G\big(x,u(x)\big)- g\big(x,u(x)\big)v(x)\big|dx  \nonumber\\
		&\leq& \int_{\mathbb{R}^N}\big|g\big(x,z(x)\big)-g\big(x,u(x)\big)\big|\ |v(x)|dx \nonumber\\
		&\leq& C_2||\xi||_{L^2(\mathbb{R}^N)}||v||,
		\end{eqnarray}
	where $\xi(x):= |g(x ,z(x))-g(x,u(x))|$ and $z(x)= u(x)+\theta(x)v(x)$, with $0\leq\theta(x)\leq 1$ given by Mean Value Theorem and $C_2>0$ is the constant given by the continuous \ embedding $ E \hookrightarrow L^2(\mathbb{R}^N)$.
	
	\hspace{0.75cm}In order to prove that $B$ is uniformly differentiable on bounded sets of $E$, given $\varepsilon>0$ it is sufficient to show there exist $\delta>0$ such that $C_2||\xi||_{L^2(\mathbb{R}^N)}\leq \varepsilon$ for all $u+v, \ v \in B_R$ with $||v||\leq \delta$. Seeking a contradiction, suppose that it is not the case, then there exist $R_0, \varepsilon_0>0$ such that for all $\delta>0$ there are $u_{\delta}+v_{\delta}, v_{\delta} \in B_{R_0}$ with $||v_{\delta}||\leq \delta$ and $C_2||\xi||_{L^2(\mathbb{R}^N)} > \varepsilon_0$. Thus, it is possible to obtain for all $n \in \mathbb{N}$ and $\delta= \dfrac{1}{n}$ functions $u_n+v_n, \ v_n \in B_{R_0}$ such that $||v_n||\leq \dfrac{1}{n}$ and $ C_2||\xi_n||_{L^2(\mathbb{R}^N)}> \varepsilon_0$, for $\xi_n(x)= |g(x ,z_n(x))-g(x,u_n(x))|$, with $z_n = u_n +\theta_nv_n$, and $0\leq \theta_n\leq 1$ depending on $u_n$ and $v_n$ as before. Since $v_n \to 0 $ in $E$, then $v_n \to 0$ in $L^2(\mathbb{R}^N)$, $v_n(x)\to 0$ almost everywhere in $\mathbb{R}^N$ and there exists $\psi \in L^2(\mathbb{R}^N)$ such that $|v_n(x)|\leq \psi(x)$ almost everywhere in $\mathbb{R}^N$. Furthermore, since $(u_n)\subset B_{R_0}$, it is bounded in $E$, then $u_n \rightharpoonup u$ in $E$  up to subsequences, then $u_n \to u$ in  $L^2_{loc}(\mathbb{R}^N)$ up to subsequences, hence $u_n(x)\to u(x)$ almost everywhere in $\mathbb{R}^N$ and fixed $B_r(0) \subset \mathbb{R}^N$ there exists $\varphi_r \in L^2(B_r(0))$ such that $|u_n(x)|\leq \varphi_r(x)$ almost everywhere in $B_r(0)$ up to subsequences. In addition, $z_n(x) \rightharpoonup u$ in $E$   up to subsequences, then $z_n \to u$ in  $L^2_{loc}(\mathbb{R}^N)$ up to subsequences, hence $z_n(x)\to u(x)$ almost everywhere in  $\mathbb{R}^N$, which implies that
	$\xi_n(x) \to 0,$
	almost everywhere in $\mathbb{R}^N$, provided that $g$ is continuous.   Moreover, in view of Remark \ref{crr2} with $p=2$, it yields
	\begin{eqnarray}\label{crl2e2}
	|\xi_n(x)|^2&\leq&2\Big[\big|g(x,z_n(x))\big|^2+\big|g(x,u_n(x))\big|^2\Big]\nonumber\\
	&\leq& 2\Big[C^2|z_n(x)|^2 + C^2|u_n(x)|^2\Big] \nonumber\\
	&\leq&2C^2\Big[2\big(|u_n(x)|^2+|v_n(x)|^2\big)+|u_n(x)|^2\Big] \nonumber\\
	&\leq& 2C^2\Big[3|u_n(x)|^2+ 2|v_n(x)|^2\Big] \nonumber\\
	&\leq& 6C^2\Big[ \varphi_r^2(x)+\psi^2(x)\Big],
	\end{eqnarray}
	almost everywhere in  $B_r(0)$. Since $\varphi_r^2 +\psi^2 \in L^1(B_r(0))$, applying Lebesgue Dominated Convergence Theorem, it yields
	\begin{equation}\label{crl2e3}
	\int_{B_r(0)}|\xi_n(x)|^2dx \to 0,
	\end{equation}
	as $n \to +\infty$.	On the other hand, since $(z_n)\subset H^1_{rad}(\mathbb{R}^N)$ and $ (u_n)\subset H^1_{rad}(\mathbb{R}^N)$ are bounded sequences, it follows that 
	$$\displaystyle\lim_{|x| \to +\infty}z_n(x)= \displaystyle\lim_{|x| \to +\infty}u_n(x)= 0,$$
	uniformly with respect to $n$, by the characterization of decay of radial functions (cf. for instance \cite{BeLi} Radial Lemma A.II). Hence, given $\varsigma>0$, there exists $r>0$ such that $|x|\geq r$ implies  $|z_n(x)|, \ |u_n(x)|\leq \varsigma$ for all $n \in \mathbb{N}$. Moreover, given $\vartheta >0$ by $(g_1)$ there exists $\varsigma >0$ sufficiently small such that $|g(x,s)| \leq \vartheta|s|$ for all $|s|\leq \varsigma$. Hence, for $r>0$ sufficiently large, it yields
	$$|g(x,z_n(x))|\leq \vartheta|z_n(x)| \text{ \ and \ }
	|g(x,u_n(x))| \leq \vartheta|u_n(x)|,$$
	for all $|x|\geq r$ and  since $(z_n)$ and $(u_n)$ are bounded in $L^2(\mathbb{R}^N)$, it yields
	\begin{eqnarray}\label{crl2e4}
	\int_{\mathbb{R}^N \backslash B_r(0)}|\xi_n(x)|^2dx &\leq& 2\int_{\mathbb{R}^N \backslash B_r(0)} \Big[\big|g(x,z_n(x))\big|^2+\big|g(x,u_n(x))\big|^2\Big]dx\nonumber\\
	&\leq& 2\vartheta\int_{\mathbb{R}^N \backslash B_r(0)}\Big(|z_n(x)|^2+|u_n(x)|^2\Big)dx\nonumber \\
	&\leq&2\vartheta\sup_n\left(||z_n||_2^2+||u_n||_2^2\right) \nonumber\\
	&\leq& C\vartheta \nonumber\\
	&<&\dfrac{1}{2}\left(\dfrac{\varepsilon_0}{C_2}\right)^2,
	\end{eqnarray}
	for $\vartheta$ sufficiently small. Therefore, combining (\ref{crl2e3}) and (\ref{crl2e4}) it follows that as $n \to +\infty$
	$$\left(\dfrac{\varepsilon_0}{C_2}\right)^2<||\xi_n||^2_{L^2(\mathbb{R}^N)}=\int_{ \mathbb{R}^N}|\xi_n(x)|^2dx \leq o_n(1) + \dfrac{1}{2}\left(\dfrac{\varepsilon_0}{C_2}\right)^2.$$
	Thus, passing to the limit as $n \to +\infty$ it yields a contradiction. Therefore, the result holds.
\end{proof}

\section{The Linking Geometry}

\hspace{0.75cm}For the linking geometry, set $S:= (\partial B_{\rho}\cap E_1)$ and
$${Q:= \{re+u_2: r\geq0, u_2 \in E_2, ||re+u_2||\leq r_1\}},$$
where $0<\rho < r_1$ are constants and $e\in E_1, ||e||=1$, is chosen suitably. Indeed, due to the strict inequality in  $(g_3)$ and from Remark \ref{crr1},  it is possible to choose  $e \in E_1$ a unitary vector given by the spectral family of operator $A$ and $\varepsilon>0$ small enough satisfying
\begin{eqnarray} \label{cre6}
1&=&||e||^2=Q_A(e) = \dfrac{1}{2}(Ae,e)_{L^2(\mathbb{R}^N)} \nonumber\\
&\leq& \dfrac{1}{2}(\sigma^+ + \varepsilon )||e||^2_2\nonumber\\
&<&\dfrac{1}{2}a_0||e||^2_2 \nonumber\\
&\leq& \dfrac{1}{2}\int_{\mathbb{R}^N}h(x)e^2(x) \ dx.
\end{eqnarray}
Choosing such an $e$, by means of (\ref{cre6}) it is possible to show that for sufficiently large $r_1>0$,  $I|_S\geq \alpha >0$ and $I|_{\partial Q}\leq 0$ hold, for some $\alpha>0$. Moreover,
$S$ and $Q$ \ ``link'' (cf. \cite{MS}). Hence, $I$ satisfies $(I_3)$ for some $\alpha>0, \ \omega=0$ and arbitrary $v \in E_2$.

\begin{lemma} \label{crl3}
	Under the hypotheses $(V_1)_r-(V_2)_r$ on $V$ and $(g_1)-(g_3)$ on $g$, $I$ satisfies $(I_3)$.	
\end{lemma}	

\begin{proof}
Since $S\subset E_1$, by Remark \ref{crr2},  for $2 < p < 2^*$ and for all $u_1 \in S$, it yields
	\begin{eqnarray}\label{crl3e1}
	I(u_1)&=&\dfrac{1}{2}||u_1||^2 - \displaystyle\int_{\mathbb{R}^N}G(x,u_1(x))dx \nonumber\\
	&\geq&\dfrac{1}{2}\rho^2- \int_{\mathbb{R}^N}\left(\dfrac{\varepsilon}{2}|u_1(x)|^2 + \dfrac{C_{\varepsilon}}{p}|u_1(x)|^p\right)dx \nonumber\\
	&\geq& \dfrac{1}{2}\rho^2 - \left(\dfrac{\varepsilon}{2}C_2^2||u_1||^2+\dfrac{C_{\varepsilon}}{p}C_p^p||u_1||^p\right)\nonumber\\
	&=& \rho^2\left[\dfrac{1}{2}\big(1-\varepsilon C_2^2\big) - \dfrac{C_{\varepsilon}}{p}C_p^p\rho^{p-2}\right]\nonumber \\
	&\geq& \rho^2(d_1-d_2) = \alpha>0,
	\end{eqnarray}
	where $\varepsilon, \rho$ are sufficiently small, such that $1>\varepsilon C_2^2$ and also
	$$d_1:= \dfrac{1}{2}\big(1-\varepsilon C_2^2\big)>\dfrac{C_{\varepsilon}}{p}C_p^p\rho^{p-2}=:d_2.$$
	Therefore, from (\ref{crl3e1}), $(I_3) \ (i)$ holds for $I$.
	
	\hspace{0,75cm}In order to prove that $I$ satisfies $(I_3) \ (ii)$ in Theorem \ref{ALT}, with $\omega = 0$, observe that $I(u)\leq 0$, for all $u \in E_2=E^-\oplus E^0$, then it suffices to show that $I(re+u)\leq 0$ for $r>0, u \in E_2$ and $||re+u||\geq r_1$, for some $r_1>0$ large enough. Arguing indirectly assume that for some sequence $(r_ne+u_n)\subset \mathbb{R}^+e \oplus E_2$ with $ ||r_ne+u_n||\to +\infty$, $I(r_ne+u_n)>0$ holds, for all $n \in \mathbb{N}$. Seeking a contradiction, set $$\tilde{u}_n := \dfrac{r_ne+u_n}{||r_ne+u_n||}= s_ne+ w_n,$$
	where $ s_n \in \mathbb{R}^+, w_n=w^-_n +w^0_n \in E_2 = E^-\oplus E^0$ and $||\tilde{u}_n||=1$. Provided that $(\tilde{u}_n)$ is bounded, up to subsequences it follows that $\tilde{u}_n \rightharpoonup \tilde{u}= se + w$ in $E$, hence $\tilde{u}_n \to u$ in $L^2_{loc}(\mathbb{R}^N)$. Then, up to subsequences, $\tilde{u}_n(x) \to \tilde{u}(x)$ almost everywhere in $\mathbb{R}^N$, $s_n \to s$ in $\mathbb{R}^+$, $w^-_n \rightharpoonup w$ in $E,$ and \ $w^0_n \to w^0$ in $E$, since $s_n$, $w_n^-$ and $w^0_n$ are also bounded, $(w^0_n) \subset E^0$ and $E^0$ is finite dimensional. Noting that $1 = ||s_ne+w_n||^2= s_n^2 +||w^-_n||^2+||w^0_n||^2$, it follows that $0\leq s^2_n\leq 1$, and it yields
	\begin{eqnarray}\label{crl3e2}
	\dfrac{I(r_ne+u_n)}{||r_ne+u_n||^2} &=& {s^2_n}||e||^2 - ||w^-_n||^2 - \int_{\mathbb{R}^N}\dfrac{G(x,r_ne(x)+u_n(x))}{||r_ne+u_n||^2}dx\nonumber\\
	&=& 2s^2_n - {1} - ||w^0_n||^2 - \int_{\mathbb{R}^N}\dfrac{G(x, r_ne(x)+u_n(x))}{||r_ne+u_n||^2}dx  >0,
	\end{eqnarray}
	hence $0<s\leq1$. 
	Moreover, from (\ref{cre6}) it is possible to choose a bounded domain $\Omega \subset \mathbb{R}^N$, such that
	$$1< \dfrac{1}{2}\displaystyle\int_{\Omega}h(x)e^2(x)dx.$$
	Hence, 
	\begin{eqnarray}\label{crl3e3}
	0&>&s^2-s^2\dfrac{1}{2}\displaystyle\int_{\Omega}h(x)e^2(x)dx\nonumber\\
	&\geq& s^2\left(1 - \dfrac{1}{2}\displaystyle\int_{\Omega}h(x)e^2(x)dx\right) - (1 + ||w^0||^2-s^2) - \dfrac{1}{2}\displaystyle\int_{\Omega}h(x)w^2(x)dx\nonumber\\
	&=& s^2\left(2 - \dfrac{1}{2}\displaystyle\int_{\Omega}h(x)e^2(x)dx\right) - 1 - ||w^0||^2 - \dfrac{1}{2}\displaystyle\int_{\Omega}h(x)w^2(x)dx.
	\end{eqnarray}
	On the other hand, from  assumptions $(g_1)-(g_2)$ and since $\tilde{u}_n$ is convergent in $L^2(\Omega)$, there exists some $\psi \in L^1(\Omega)$ such that $$\left|\dfrac{G(\ \cdot \ ,r_ne(\cdot)+u_n(\cdot))}{||r_ne+u_n||^2}\right|\leq r_{\infty}|\tilde{u}_n(\cdot)|^2\leq \psi(\cdot)\in L^1(\Omega).$$
	Moreover, provided that $||r_ne + u_n||\to +\infty$, and $\tilde{u}_n(x)\to \tilde{u}(x)\not= 0$, almost everywhere in $supp(\tilde{u})$,  it follows that $u_n(x)=\tilde{u}_n(x)||r_ne + u_n(x) ||\to +\infty$ almost everywhere in $supp(\tilde{u})$, as $n\to +\infty$, hence
	$$\dfrac{G(x,r_ne(x)+u_n(x))}{||r_ne+u_n||^2} = \dfrac{G(x,\tilde{u}_n(x)||r_ne+u_n||)\tilde{u}^2_n(x)}{\tilde{u}^2_n(x)||r_ne+u_n||^2} \to \dfrac{1}{2}h(x)\tilde{u}^2(x),$$\\
	almost everywhere in $supp(\tilde{z})$ as $n \to +\infty$. Note that, $supp(\tilde{u}) \not=\emptyset$, because $\tilde{u} = se+w$, with $supp(e)\not= \emptyset$ and $(e,w)_{L^2(\mathbb{R}^N)}=0$. Thus, by Lebesgue Dominated Convergence Theorem,
	
	$$\int_{\Omega}\dfrac{G(x,r_ne(x)+u_n(x))}{||r_ne+u_n||^2}dx \to \dfrac{1}{2}\int_{\Omega}h(x)\big(se(x) + w(x)\big)^2dx,$$\\
	as $n \to +\infty.$ From (\ref{crl3e2}) one has
	$$2s^2_n - {1} - ||w^0_n||^2 - \int_{\Omega}\dfrac{G(x, r_ne(x)+u_n(x))}{||r_ne+u_n||^2}dx  >0.$$	
	Passing to the limit as $n \to +\infty$, it yields
	
	\begin{eqnarray}\label{crl3e4}
	0&\leq& 2s^2 - {1} - ||w^0||^2 - \dfrac{1}{2}\int_{\Omega}h(x)\Big(s^2e^2(x) + w^2(x)\Big)dx \nonumber\\ &=& s^2\left(2 - \dfrac{1}{2}\displaystyle\int_{\Omega}h(x)e^2(x)dx\right) - {1} - ||w^0||^2 - \dfrac{1}{2}\displaystyle\int_{\Omega}h(x)w^2(x)dx,\\
	\nonumber
	\end{eqnarray}
	which is contrary to (\ref{crl3e3}). Therefore the result holds.
\end{proof}

\section{The Boundedness of Cerami Sequences}

\begin{lemma}	\label{crl4}
	Suppose that $V$ satisfies $(V_1)_r-(V_2)_r$ and  $g$ satisfies $(g_1)-(g_4)$, then $I$ satisfies $(I_4)$.	
\end{lemma}

\begin{proof}
	Let \ $b>0$ be \ an \ arbitrary \ constant, \ and \ take \ $(u_n) \subset I^{-1}([c-b,c+b])$ \ such that $\left(1+||u_n||\right)||I'(u_n)||\to0$, it is necessary to show that $(u_n)$ is bounded. Suppose by contradiction that $||u_n||\to +\infty$, up to subsequences. Setting $\tilde{u}_n:=\dfrac{u_n}{||u_n||}$, it is bounded, hence $\tilde{u}_n \rightharpoonup \tilde{u}$ in $E$ and $\tilde{u}_n \to \tilde{u}$ in $L^\beta(\mathbb{R}^N)$, for $\beta \in (2,2^*)$, due to the compact embeddings previously mentioned (cf. \cite{S} and \cite{BeLi}). Writing $u_n=u^+_n+u^-_n+u^0_n \in E^+\oplus E^-\oplus E^0$, by choice of $u_n$ it satisfies
	\begin{eqnarray}\label{crl4e1}
	o_n(1)&=& I'(u_n)\dfrac{u^+_n}{||u_n||^2} \nonumber\\
	&=&\dfrac{1}{||u_n||}I'(u_n)\tilde{u}^+_{n}\nonumber\\
	&=&||\tilde{u}^+_{n}||^2 - \int_{\mathbb{R}^N}\dfrac{g(x,u_n(x))}{u_n(x)}\tilde{u}_n(x)\tilde{u}^+_{n}(x)dx.
	\end{eqnarray}
	and
	\begin{eqnarray}\label{crl4e2}
	o_n(1)&=& I'(u_n)\dfrac{u^-_n}{||u_n||^2} \nonumber\\
	&=&\dfrac{1}{||u_n||}I'(u_n)\tilde{u}^-_{n}\nonumber\\
	&=&||\tilde{u}^-_{n}||^2 - \int_{\mathbb{R}^N}\dfrac{g(x,u_n(x))}{u_n(x)}\tilde{u}_n(x)\tilde{u}^-_{n}(x)dx. \\
	\nonumber
	\end{eqnarray}
	Subtracting (\ref{crl4e2}) from (\ref{crl4e1}), and using that $1 = ||\tilde{u}_n^+||^2 + ||\tilde{u}_n^-||^2 + ||\tilde{u}_n^0||^2$, it yields\\
	\begin{equation}\label{crl4e3}
	o_n(1)=1 - ||\tilde{u}^0_{n}||^2 - \int_{\mathbb{R}^N}\dfrac{g(x,u_n(x))}{u_n(x)}\Big[(\tilde{u}^+_{n}(x))^2-(\tilde{u}^-_{n}(x))^2\Big]dx.\\
	\end{equation}\\
	Provided that $(\tilde{u}^0_n) \subset E^0$, which is finite dimensional, then the weak convergence implies that $\tilde{u}^0_n \to \tilde{u}^0$ in $E$. Furthermore, since $\tilde{u}_n\to\tilde{u}$ in $L_{loc}^2(\mathbb{R}^N)$, fixed $B_r(0) \subset \mathbb{R}^N$ \ there \ exist \ $\psi_r^+, \psi_r^- \in L^2(B_r(0))$ \ such \ that $|\tilde{u}^+_n(x)|\leq \psi_r^+(x)$  and $|\tilde{u}^-_n(x)|\leq \psi_r^-(x)$, almost everywhere in $B_r(0)$, hence from Remark \ref{crr2} it follows that
	$$\left|\dfrac{g(\cdot,u_n(\cdot))}{u_n(\cdot)}\left[\big(\tilde{u}^+_{n}(\cdot)\big)^2-\big(\tilde{u}^-_{n}(\cdot)\big)^2\right]\right|\leq C\left[\big(\psi_r^+(\cdot)\big)^2+\big(\psi_r^-(\cdot)\big)^2\right] \in L^1(B_r(0)).$$
	Since $\tilde{u}_n \to \tilde{u}$ in $L^2(B_r(0))$, $u_n(x)\to +\infty$, for all $x \in B_r(0)$ such that $\tilde{u}(x)\not=0$, then from $(g_2)$ it follows that
	\begin{equation*}
	\dfrac{g(x,u_n(x))}{u_n(x)}\left[\big(\tilde{u}^+_{n}(x)\big)^2-\big(\tilde{u}^-_{n}(x)\big)^2\right] \to h(x)\left[\big(\tilde{u}^+(x)\big)^2-\big(\tilde{u}^-(x)\big)^2\right],
	\end{equation*}
	as $n \to +\infty$, for all $x \in B_r(0)$ such that $\tilde{u}(x)\not=0$. Therefore, by Lebesgue Dominated Convergence Theorem one has
	\begin{equation} \label{crl4e4}
	\int_{B_r(0)}\dfrac{g(x,u_n(x))}{u_n(x)}\Big[(\tilde{u}^+_{n}(x))^2-(\tilde{u}^-_{n}(x))^2\Big]dx \to \int_{B_r(0)}h(x)\left[\big(\tilde{u}^+(x)\big)^2-\big(\tilde{u}^-(x)\big)^2\right]dx.
	\end{equation}
	Moreover, since $ (u_n)\subset H^1_{rad}(\mathbb{R}^N)$ it follows that 
	$$\displaystyle\lim_{|x| \to +\infty}u_n(x)= 0,$$
	uniformly with respect to $n$. Hence, given
	$\delta>0$, there exists $r>0$ such that $|x|\geq r$ implies  $|u_n(x)|\leq \delta$ for all $n \in \mathbb{N}$. In addition, given $\varepsilon >0$ by $(g_1)$ there exists $\delta >0$ sufficiently small such that $|g(x,s)| \leq \varepsilon|s|$ for all $|s|\leq \delta$. Hence, given $\varepsilon >0$, for $r>0$ sufficiently large, it yields
	$$|g(x,u_n(x))| \leq \varepsilon|u_n(x)|,$$
	for all $|x|\geq r$ and  since $(u_n)$ is bounded in $L^2(\mathbb{R}^N)$, it yields
	\begin{eqnarray} \label{crl4e5}
	\int_{\mathbb{R}^N \backslash B_r(0)}\dfrac{g(x,u_n(x))}{u_n(x)}\Big[(\tilde{u}^+_{n}(x))^2-(\tilde{u}^-_{n}(x))^2\Big]dx	&\leq& \varepsilon\int_{\mathbb{R}^N \backslash B_r(0)}\Big[(\tilde{u}^+_{n}(x))^2+(\tilde{u}^-_{n}(x))^2\Big]dx \nonumber\\
	&\leq&2\varepsilon \sup_n||u_n||_2^2 \nonumber\\
	&\leq&C\varepsilon.
	\end{eqnarray}
	Thus, combining (\ref{crl4e4}) and (\ref{crl4e5}) it follows that
	\begin{eqnarray} \label{crl4e6}
	&&\int_{\mathbb{R}^N}\dfrac{g(x,u_n(x))}{u_n(x)}\Big[(\tilde{u}^+_{n}(x))^2-(\tilde{u}^-_{n}(x))^2\Big]dx \nonumber\\
	&=& \int_{B_r(0)}\dfrac{g(x,u_n(x))}{u_n(x)}\Big[(\tilde{u}^+_{n}(x))^2-(\tilde{u}^-_{n}(x))^2\Big]dx \nonumber\\
	&+& \int_{\mathbb{R}^N \backslash B_r(0)}\dfrac{g(x,u_n(x))}{u_n(x)}\Big[(\tilde{u}^+_{n}(x))^2-(\tilde{u}^-_{n}(x))^2\Big]dx \nonumber\\
	&=& \int_{B_r(0)}h(x)\left[\big(\tilde{u}^+(x)\big)^2-\big(\tilde{u}^-(x)\big)^2\right]dx \nonumber\\
	&+& C\varepsilon + o_n(1),
	\end{eqnarray}
	as $n \to +\infty$. Hence, passing to the limit as $n \to +\infty$ and $\varepsilon \to 0^+$, it implies that
	\begin{equation} \label{crl4e7}
	\int_{\mathbb{R}^N}\dfrac{g(x,u_n(x))}{u_n(x)}\Big[(\tilde{u}^+_{n}(x))^2-(\tilde{u}^-_{n}(x))^2\Big]dx \to \int_{\mathbb{R}^N}h(x)\left[\big(\tilde{u}^+(x)\big)^2-\big(\tilde{u}^-(x)\big)^2\right]dx,
	\end{equation}
	as $n \to +\infty$. Therefore, passing to the limit in (\ref{crl4e3}) as $n \to +\infty$, it yields
	\begin{equation}\label{crl4e8}
	\int_{\mathbb{R}^N}h(x)\left[\big(\tilde{u}^+(x)\big)^2-\big(\tilde{u}^-(x)\big)^2\right]dx = 1 - ||\tilde{u}^0||^2.
	\end{equation}
	On the other hand, given $\varphi \in C_0^{\infty}(\mathbb{R}^N)$ and setting $supp(\varphi):=K$, since $\tilde{u}_n \to \tilde{u}$ in $L^2(K)$, in virtue of similar arguments, by applying Lebesgue Dominated Convergence Theorem it follows that
	$$\int_{K}\dfrac{g(x,u_n(x))}{u_n(x)}\tilde{u}_n(x)\varphi(x)dx = \int_{K}h(x)\tilde{u}_n(x)\varphi(x)dx + o_n(1),$$
	as $n \to +\infty$. Hence, it yields
	\begin{eqnarray}\label{crl4e9}
	o_n(1)&=& \dfrac{I'(u_n)\varphi}{||u_n||} \nonumber\\
	&=& \dfrac{Q'_A(u_n)\varphi}{||u_n||}- \int_{K}\dfrac{g(x,u_n(x))}{u_n(x)}\tilde{u}_n(x)\varphi(x)dx \nonumber\\
	&=& (A\tilde{u}_n,\varphi)_{L^2(\mathbb{R}^N)} - \int_{K}h(x)\tilde{u}_n(x)\varphi(x)dx + o_n(1) \nonumber\\
	&=& (\mathcal{O}\tilde{u}_n,\varphi)_{L^2(\mathbb{R}^N)} + o_n(1) \nonumber\\
	&=& (\mathcal{O}\tilde{u},\varphi)_{L^2(\mathbb{R}^N)} + o_n(1).
	\end{eqnarray}
	Due to (\ref{crl4e9}), if $\tilde{u} \not= 0$, it is an eigenvector of $\mathcal{O}$, with eigenvalue $0$. Nevertheless, from $(g_4)$, $0\notin \sigma_p(\mathcal{O})$ and hence $\tilde{u}=0$. It means that $\tilde{u}^+=\tilde{u}^-=\tilde{u}^0 = 0$ and thus, (\ref{crl4e8}) yields a contradiction. Therefore, $(u_n)$ is bounded and the result holds.
\end{proof}

\hspace{0.75cm}Finally it is possible to prove the main result of this section.\\

\hspace{-0.5cm}\textit{Proof of Theorem \ref{crt1}.}
Provided that $I$ satisfies all assumptions $(I_1)-(I_4)$ in Theorem \ref{ALT}, it ensures a critical point  $u \in E$ of $I$, with $I(u)=c\geq\alpha>0$, hence $u$ is a non-trivial critical point of $I:E\to \mathbb{R}$. It implies that $I'(u)v = 0$, for all $v \in H^1_{rad}(\mathbb{R}^N)$. Nevertheless, the Principle of Symmetric Criticality \cite{Pal} implies that $I'(u)v = 0$ for all $v \in H^1(\mathbb{R}^N)$, namely, $u$ is a critical point of $I$ as a functional defined on the whole $H^1(\mathbb{R}^N)$. Since $I\in C^1(H^1(\mathbb{R}^N),\mathbb{R})$, it yields that $u$ is a nontrivial weak solution of $(P_r)$. In addition, since $u \in E$, it is a radial weak solution. 
\vspace{-0.5cm}
\begin{flushright}$\bf{\square}$\end{flushright}

\bigskip
 
 \hspace{0.75cm}Note that setting $\bar{g}(x,s)=0$ for $s<0$ and $\bar{g}(x,s)=g(x,s)$ for $s\geq0$ and repeating the arguments, it is possible to obtain a positive solution for problem (\ref{P_r}).

\end{document}